\documentclass{amsart}
\usepackage{amsthm}
\usepackage{amsmath, ifthen, amssymb}
\usepackage{latexsym}
\usepackage{amsfonts}

\begin{document}

\newcommand{\nat}{\omega}
\newcommand{\Tr}{{\rm Tr}}
\newcommand{\WF}{{\rm WF}}
\newcommand{\ie}{{\rm i.e.,} }
\newcommand{\cf}{{\rm cf.~}}
\newcommand{\n}{\ensuremath{n \in \nat}}
\newcommand{\iin}{\ensuremath{i \in \nat}}
\newcommand{\s}{\ensuremath{s \in \nat}}
\newcommand{\ep}{\ensuremath{\varepsilon}}
\newcommand{\eq}{\ensuremath{\Longleftrightarrow}}
\newcommand{\del}{\ensuremath{\Delta^1_1}}
\newcommand{\pii}{\ensuremath{\Pi^1_1}}
\newcommand{\G}{\ensuremath{\Gamma}}
\newcommand\tboldsymbol[1]{%
\protect\raisebox{0pt}[0pt][0pt]{%
$\underset{\widetilde{}}{\boldsymbol{#1}}$}\mbox{\hskip 1pt}}
\newcommand{\boldpii}{\ensuremath{\tboldsymbol{\Pi}^1_1}}
\newcommand{\om}{\ensuremath{\omega}}
\newcommand{\R}{\ensuremath{\mathbb R}}
\newcommand{\ds}{\ensuremath{\displaystyle}}
\renewcommand{\qedsymbol}{$\dashv$}
\newcommand{\ca}[1]{\ensuremath{\mathcal{#1}}}
\newcommand{\set}[2]{\ensuremath{\{#1 \hspace{0.3mm} \mid \hspace{0.3mm} #2\}}}
\newcommand{\tu}[1]{\textup{#1}}
\newcommand{\bg}{\tboldsymbol{\G}}
\newcommand{\pntcl}{\bg}
\newcommand{\norm}[1]{\ensuremath{\|#1\|}}
\newcommand{\Mem}{\textrm{Mem}}
\newcommand{\Empt}{\textrm{Empt}}
\newcommand{\Ball}{\textrm{UB}}
\newcommand{\Det}[1]{\ensuremath{\textrm{Det}}(#1)}
\newcommand{\fixp}[1]{\ensuremath{\textrm{Fix}_{#1}}}
\newcommand{\ZF}{\textbf{ZF}}
\newcommand{\DC}{\textbf{DC}}
\newcommand{\ZFDC}{\textbf{ZFDC}}
\newcommand{\ZFC}{\textbf{ZFC}}
\newcommand{\AC}{\textbf{AC}}
\newcommand{\REFL}{{\rm REFL}}
\newcommand{\BC}{{\rm BC}}
\newcommand{\WCC}{{\rm WCC}}
\newcommand{\pr}{\textrm{pr}}

\newtheorem{theorem}{Theorem}
\newtheorem{lemma}[theorem]{Lemma}
\newtheorem{definition}[theorem]{Definition}
\newtheorem{proposition}[theorem]{Proposition}
\newtheorem{corollary}[theorem]{Corollary}
\newtheorem{question}[theorem]{Question}
\newtheorem{remark}[theorem]{Remark}
\newtheorem{remarks}[theorem]{Remarks}
\newtheorem{example}[theorem]{Example}
\newtheorem{examples}[theorem]{Examples}

\title{Choice free Fixed Point Property in separable Banach spaces}

\author[V. Gregoriades]{Vassilios Gregoriades}
\address{Technische Universit\"{a}t Darmstadt,
Fachbereich Mathematik,
Arbeitsgruppe Logik,
Schlo{\ss}gartenstra{\ss}e 7,
64289 Darmstadt
Germany}
\email{gregoriades [at] mathematik [dot] tu-darmstadt [dot] de}

\date{}

\keywords{minimal invariant sets, non-expansive mappings, fixed point property, Axiom of Choice, effective descriptive set theory.}

\thanks{This a pre-publication draft of the work published in the \textbf{Proceedings of the American Mathematical Society 143 (2015), no. 5, 2143-2157}.\newline
The author would like to thank {\sc U. Kohlenbach}, {\sc A. Kreuzer}, {\sc H. Mildenberger} and {\sc C. Poulios} for helpful discussions, remarks and suggestions. The author would also like to give special thanks to the referee of this article, for their valuable suggestions and for providing some very interesting material which can be found at the last pages of the article.}

\subjclass[2010]{47H10, 03E15, 54H05, 54H25}

\maketitle

\begin{abstract}
We show that the standard approach of minimal invariant sets, which applies Zorn's Lemma and is used to prove fixed point theorems for non-expansive mappings in Banach spaces can be applied without any reference to the full Axiom of Choice when the given Banach space is separable. Our method applies results from classical and effective descriptive set theory.
\end{abstract}

\subsection*{Introduction.} A mapping $T:  X \to Y$ between Banach spaces is \emph{non-expansive} if $\norm{Tx - Ty}_Y \leq \norm{x - y}_X$, for all $x, y \in X$. A Banach space $X$ has the \emph{fixed point property} if for all non-empty convex and weakly compact $F \subseteq X$ and for all non-expansive mappings $T: F \to F$ there exists an $x \in F$ with $T(x) = x$, \ie \emph{$x$ is a fixed point of $T$}. A set $A$ is \emph{T-invariant} if $T[A] \subseteq A$ .

It is well known that Banach spaces which are uniformly convex or have a ``normal" structure have the fixed point property. In fact in these cases one can give a constructive proof of the existence of a fixed point, which in particular does not use Zorn's Lemma, see for example \cite{goehde_zum_prinzip_der_kontraktiven_abbildung} and \cite{kirk_an_abstract_fixed_point_theorem_for_nonexpansive_mappings}.

Nevertheless the standard technique for proving that a Banach space $X$ has the fixed point property is to show -given $F$ and $T$ as above- that there exists a non-empty, $T$-invariant, convex, weakly compact $L \subseteq F$ which is minimal with respect to these properties, \cf \cite{goebel_kirk_some_problems_in_metric_fixed_point_theory}. Then one shows that $L$ has zero diameter and so $L$ must be a singleton, say $\{x\}$. Since $L$ is $T$-invariant we have that $x$ is a fixed point of $T$.

The typical way of verifying the existence of a minimal set $L$ as above is by applying Zorn's Lemma. Our main aim is to show that, in the setting of separable Banach spaces, it is possible to obtain such a minimal set without the Axiom of Choice (\AC), \cf Theorem \ref{theorem minimal weakly compact convex}. We will do this by applying results from classical and effective descriptive set theory. At the end of this article we present a Banach space theoretic approach for proving Theorem \ref{theorem minimal weakly compact convex}, which does not use effective theory.

Let us see first exactly where the Axiom of Choice is used in order to prove the existence of a minimal set as above. This is an instantiation of the derivation of Zorn's Lemma from \AC.

A partially ordered space $(\mathbb{P},\leq)$ is \emph{inductive} if every linearly ordered subset of $\mathbb{P}$ has a least upper bound and a mapping $f: \mathbb{P} \to \mathbb{P}$ is \emph{expansive} if $x \leq f(x)$ for all $x \in \mathbb{P}$.

\begin{theorem}[Zermelo's Fixed Point Theorem]

\label{theorem Zermelo Fixed Point}

For every inductive space $(\mathbb{P},\leq)$ and for every expansive function $f: \mathbb{P} \to \mathbb{P}$ there exists some $x^* \in \mathbb{P}$ with $$f(x^*) = x^*.$$
\end{theorem}

Now consider a Banach space $X$, a non-empty convex weakly compact $F \subseteq X$ and a non-expansive mapping $T: F \to F$. We define the set
\[
\mathbb{P} = \set{K \subseteq F}{K \ \textrm{is non-empty, convex, weakly compact and $T$-invariant}}
\]
and we consider the relation $\leq$ of the inverse inclusion, \ie
\[
K \leq L \eq L \subseteq K
\]
for all $K, L \in \mathbb{P}$. Define also the \emph{strict part} $<$ of $\leq$ by
\[
K < L \eq K \leq L \ \& \ K \neq L
\]
for all $K, L \in \mathbb{P}$. It is clear that a minimal non-empty, convex, weakly compact and $T$-invariant $L \subseteq F$ is exactly a maximal point of $(\mathbb{P},\leq)$. It is not hard to verify that $(\mathbb{P},\leq)$ is inductive. Let us assume towards contradiction that $(\mathbb{P},\leq)$ does not have a maximal point. Using the Axiom of Choice we obtain a function $f: \mathbb{P} \to \mathbb{P}$ such that $K < f(K)$ for all $K \in \mathbb{P}$. Thus the function $f$ is expansive without a fixed point, contradicting Zermelo's Fixed Point Theorem.\footnote{There is a bit of trickery here since in order to derive Zorn's Lemma from the Axiom of Choice one applies Zermelo's Fixed Point Theorem to the space $(C(\mathbb{P}),\subseteq)$, where
\[
C(\mathbb{P}) = \set{\ca{S} \subseteq \mathbb{P}}{\ca{S} \ \textrm{is $\leq$-linearly ordered}}.
\]
The reason for this lies in the hypothesis of the statement of Zorn's Lemma: one starts with a space $(\mathbb{P},\leq)$, whose every linearly ordered subset has an upper bound but not necessarily a \emph{least} upper bound, \ie the space that we start with is not necessarily inductive. So we go a level up to the space $(C(\mathbb{P}),\subseteq)$, which is inductive, and we apply Zermelo's Fixed Point Theorem there: if the conclusion of Zorn's Lemma were not true, then there would be no maximal linearly ordered subset of $(\mathbb{P},\leq)$ and so there would be an expansive function $$\pi: (C(\mathbb{P}),\subseteq) \to (C(\mathbb{P}),\subseteq)$$ without a fixed point contradicting Zermelo's Fixed Point Theorem. In the case however where $(\mathbb{P},\leq)$ is inductive, as it is in our case, one can avoid the reference to $(C(\mathbb{P}),\subseteq)$ and apply Zermelo's Fixed Point Theorem directly to $(\mathbb{P},\leq)$ the way we have just described: if $(\mathbb{P},\leq)$ had no maximal point then there would be an expansive mapping $f: (\mathbb{P},\leq) \to (\mathbb{P},\leq)$ without a fixed point.}

We will show that one can actually obtain the preceding function $f$ without appealing to the Axiom of Choice. This will be an application of the \emph{uniformization property} (see below for the definition) of a certain pointclass of sets, namely the class \boldpii \ of coanalytic sets in Polish spaces, (this is Kondo's theorem that we state below). The challenge is to show that the preceding set $\mathbb{P}$ and the relation $<$ are in fact \boldpii \ subsets of some Polish spaces.\footnote{The fact that in our case the space $(\mathbb{P},\leq)$ is inductive is crucial for our purposes, for otherwise -in the light of the preceding footnote- we would have to show that $C(\mathbb{P})$ and $\subsetneq$ are \boldpii \ subsets of some Polish spaces. The latter however seems far from easy to achieve, if true at all, since the definition $(C(\mathbb{P}),\subseteq)$ is one level higher than that of $(\mathbb{P},\leq)$.}$^{,}$\footnote{Here it is worth noting that in  \cite{fuchssteiner_iterations_and_fixpoints} one can also find a method of eliminating the reference to \AC \ in favor of Zermelo's Fixed Point Theorem and from this one can derive applications to the fixed point property. Nevertheless the normal structure of the space is assumed, see p. 77.}

It is not hard to see that the preceding method extends to pointclasses other than \boldpii \ and to properties other than that of weak compactness. So we shall describe the general framework that we are working in, and derive our main result (Theorem \ref{theorem minimal weakly compact convex}) from Lemma \ref{lemma general minimal with respect to a good property}, which is stated in a more abstract context.

We point out that all our statements and proofs are given in the context of the \ZFDC \ theory, \ie the Zermelo-Fraenkel set theory (\ZF) with Dependent Choices (\DC). In particular the theorems that we invoke are provable in \ZFDC. It would nevertheless be interesting to check the validity of the following results in weaker theories. Kondo's theorem for example can be proved in the theory $\Pi^1_1$ -  $\textsf{CA}_0$, \cf \cite{simpson_subsystems_of_second_order_arithmetic}.

As mentioned above in our proofs we employ some tools from effective descriptive set theory. For a detailed exposition of the subject the reader can refer to Chapter 3 in \cite{yiannis_dst}.

Before proceeding we state a question in effective theory which has a classical, (\ie non-effective) application to the fixed point property. Suppose that \ca{X} is a separable Banach space with the fixed point property, $F$ is a non-empty weakly compact subset of \ca{X} and that $T : F \to\ F$ is a non-expansive mapping. It is not hard to verify that the (non-empty) set of fixed points of $T$,
\[
\fixp{T} = \set{x \in F}{T(x)=x}
\]
is a weakly closed subset of $F$ and therefore it is weakly compact. (Here we use the Hahn-Banach Theorem in separable Banach spaces, which is provable in $\textsf{WKL}_0$, a much weaker theory than \ZFDC, \cf \cite{simpson_subsystems_of_second_order_arithmetic}.) Assume moreover that  \ca{X} is recursively presented, $F$ is a \del \ set and that $T$ is \del-recursive. It would be interesting to see if \fixp{T} contains a \del \ member. Since \fixp{T} is easily a \del \ set the latter is reduced to the following.

\begin{question}
\label{question weakly compact contains hyp member}
Suppose that \ca{X} is a recursively presented Banach space and that $K$ is a non-empty weakly compact $\del(\alpha)$ subset of \ca{X} for some $\alpha \in \ca{N}$. Does $K$ contain a $\del(\alpha)$ point?
\end{question}

With the help of \del \ points one can derive the existence of Borel-measurable choice functions, \cf 4D.4 (the \del-uniformization criterion) and 4D.6 (the strong \del-selection principle) in \cite{yiannis_dst}. In particular if the preceding question has an affirmative answer then using the \del-uniformization criterion (or the strong \del-selection principle) one would be able to extract fixed points in a Borel-uniform way.\smallskip

\textbf{Is it true?} Suppose that \ca{Z} is a Polish space, \ca{X} is a separable Banach space which has the fixed point property and that $F$ is non-empty weakly compact subset of \ca{X}. If $T: \ca{Z} \times F \to F$ is a Borel-measurable function for which the function $$T_z: F \to F: T_z(x) = T(z,x)$$ is non-expansive for all $z \in \ca{Z}$, then there exists a Borel-measurable function $$f: \ca{Z} \to F$$ such that $f(z)$ is a fixed point of $T_z$ for all $z \in \ca{Z}$.\smallskip

We now proceed to the necessary definitions. We will often identify relations with sets and write $P(x)$ instead of $x \in P$. We also identify the first infinite ordinal number \om \ with the set of natural numbers.

\begin{definition}

\label{definition uniformization property and everything else}

\normalfont

Suppose that \ca{X} and \ca{Y} are Polish spaces and that $P$ is a subset of $\ca{X} \times \ca{Y}$. Define the set
\[
\exists^{\ca{Y}}P = \set{x \in \ca{X}}{(\exists y)P(x,y)}.
\]
A set $P^*$ \emph{uniformizes $P$} if $P^* \subseteq P$ and for all $x \in \exists^\ca{Y}P$ then there exists a unique $y \in \ca{Y}$ such that $P^*(x,y)$. In other words $P^*$ is the graph of a function $f$ such that $P(x,f(x))$ for all $x \in \exists^\ca{Y}P$.

By the term \emph{pointclass} we mean an arbitrary collection of sets in Polish spaces.

A pointclass \pntcl \ has the \emph{uniformization property} if for all Polish spaces \ca{X} and \ca{Y} and all sets $P \subseteq \ca{X} \times \ca{Y}$ in \pntcl \ there is a $P^*$ in \pntcl \ which uniformizes $P$. The pointclass \pntcl \ has the \emph{semi-uniformization property} if the preceding set $P^*$ is not necessarily a member of \pntcl, \ie if for all Polish spaces \ca{X} and \ca{Y} and all $P \subseteq \ca{X} \times \ca{Y}$ in \pntcl \ there is a $P^*$ which uniformizes $P$.

Suppose that $R(x_1,\dots,x_n)$ is an $n$-ary relation, $\ca{X}_1, \dots, \ca{X}_n$ are Polish spaces and that $P$ is a subset of $\ca{X}_1 \times \dots \ca{X}_n$. We say that the pointclass \pntcl \ \emph{computes $R$ on $P$} if there exists a set $R_{\Small{\pntcl}}$ in \pntcl \ such that
\[
R(x_1,\dots,x_n) \eq R_{\Small{\pntcl}}(x_1,\dots,x_n)
\]
for all $(x_1,\dots,x_n) \in P$.

A pointclass \pntcl \ is closed under \emph{continuous substitution} if for all continuous functions $f: \ca{X} \to \ca{Y}$ between Polish spaces and all sets $P\subseteq \ca{Y}$ in \pntcl \ the set $f^{-1}[P]$ is in \pntcl \ as well. We say that \pntcl \ is closed under \emph{logical conjunction} if for all sets $P, Q \subseteq \ca{X}$ in \pntcl \ the set $R \subseteq \ca{X}$ defined by
\[
R(x) \iff P(x) \ \& \ Q(x)
\]
is in \pntcl \ as well. In other words closure under logical conjunction means closure under finite intersections of subsets of the same space. Similarly one defines closure under logical disjunction. A pointclass \pntcl \ is \emph{good} if it is closed under continuous substitution and under the logical conjunction $\&$ and disjunction $\vee$.
\end{definition}

\begin{theorem}[Kondo \cf \cite{kondo_uniformisation} and 4E.4 \cite{yiannis_dst}]
\label{theorem Kondo}
The pointclass \boldpii \ has the uniformization property.
\end{theorem}
We also mention that the von Neumann Selection Theorem (\cf \cite{von_neumann_on_rings_of_operators_reduction_theory} and 4E.9 \cite{yiannis_dst}) implies that the pointclass $\tboldsymbol{\Sigma}^1_2$ has the semi-uniformization property.
\begin{lemma}
\label{lemma inductive space maximal point}
Suppose that \pntcl \ is a good pointclass which has the semi-uniformization property, \ca{X} is a Polish space, $\mathbb{P} \subseteq \ca{X}$ is non-empty in \pntcl \ and that $\leq$ is a partial ordering on $\mathbb{P}$ such that its strict part $<$ is computed by \pntcl \ on $\mathbb{P} \times \mathbb{P}$. If the space $(\mathbb{P},\leq)$ is inductive then it has a maximal point.
\end{lemma}

\begin{proof}
Define $R \subseteq \ca{X} \times \ca{X}$ by
\[
R(x,y) \eq x, y \in \mathbb{P} \ \& \ x < y.
\]
It is not hard to verify that $R$ is in \pntcl \ and that $R \subseteq \mathbb{P} \times \mathbb{P}$. Assume towards contradiction that $(\mathbb{P},\leq)$ does not have a maximal point. Then for all $x \in \mathbb{P}$ there is $y$ such that $R(x,y)$. We uniformize $R$ by $R^*$ and we notice that $\exists^\ca{X}R^* = \exists^\ca{X}R = \mathbb{P}$. Thus $R^*$ is the graph a function $f: \mathbb{P} \to \mathbb{P}$. Since $R(x,f(x))$ we have that $x < f(x)$ for all $x \in \mathbb{P}$. It follows that $f$ is expansive with no fixed point, contradicting Zermelo's Fixed Point Theorem, since $(\mathbb{P},\leq)$ is inductive.
\end{proof}

\subsection*{The Effros-Borel Space.} Suppose that \ca{X} is a Polish space. We denote by $F(\ca{X})$ the family of all closed subsets of \ca{X}. For al open $U \subseteq \ca{X}$ we consider the sets of the form
\[
\ca{A}_U= \set{C \in F(\ca{X})}{C \cap U \neq \emptyset}.
\]
We denote by \ca{S} the $\sigma$-algebra generated from the family $\set{\ca{A}_U}{U \subseteq \ca{X}, \ \textrm{open}}$. The \emph{Effros-Borel space} is the measurable space $(F(\ca{X}),\ca{S})$. A well-known theorem states that there is a topology $\ca{T}$ on $F(\ca{X})$ such that: (a) the space $(F(\ca{X}),\ca{T})$ is a Polish space and (b) the $\ca{T}$-Borel subsets of $F(\ca{X})$ are exactly the members of \ca{S}, \cf \cite{kechris_classical_dst} Section 12.C.

It is easy to verify that the relations $\Mem \subseteq \ca{X} \times F(\ca{X})$ and $\Empt \subseteq F(\ca{X})$ defined by
\[
\Mem(x,F) \eq x \in F
\]
and
\[
\Empt(F) \eq F = \emptyset
\]
for $x \in \ca{X}$ and $F \in F(\ca{X})$, are Borel. To see this consider a countable basis $(U_n)_{\n}$ for the topology of \ca{X} and notice that
\begin{eqnarray*}
x \in F &\eq& (\forall n)[x \in U_n \ \longrightarrow \ F \cap U_n \neq \emptyset]\\
        &\eq& (\forall n)[x \not \in U_n \ \vee \ F \in \ca{A}_{U_n}].
\end{eqnarray*}
Therefore $$\Mem = \cap_{n \in \om}([\ca{X} \setminus U_n \times F(\ca{X})] \cup [\ca{X} \times \ca{A}_{U_n}]).$$ The set $[\ca{X} \setminus U_n \times F(\ca{X})] \cup [\ca{X} \times \ca{A}_{U_n}]$ is evidently a Borel subset of $\ca{X} \times F(\ca{X})$.

Moreover
\begin{eqnarray*}
F \neq \emptyset &\eq& (\exists n)[F \cap U_n \neq \emptyset]\\
                 &\eq& (\exists n)[F \in \ca{A}_{U_n}].
\end{eqnarray*}
So $\Empt = \cap_{n} (F(\ca{X}) \setminus \ca{A}_{U_n})$.

The Effros-Borel space admits a selection theorem.

\begin{theorem}[Kuratowski, Ryll-Nardzewski, c.f. \cite{kechris_classical_dst} and \cite{kuratowski_ryll-nardzewski}]

\label{theorem selection in F(X)}

For every Polish space \ca{X} there is a sequence of Borel-measurable functions $$d_n: F(\ca{X}) \to \ca{X}, \quad \n,$$ such that for all non-empty $F \in F(\ca{X})$ the sequence $(d_n(F))_{\n}$ is contained in $F$ and is dense in $F$, \ie for all open $U \subseteq \ca{X}$ with $U \cap F \neq \emptyset$ there is some \n \ such that $d_n(F) \in U \cap F$.

\end{theorem}

\begin{lemma}[\cf \cite{kechris_classical_dst} Exercise 12.11]

\label{lemma less-or-equal is in pii}

For every Polish space \ca{X} the relation $\leq \subseteq F(\ca{X}) \times F(\ca{X})$ defined by
\[
K \leq L \eq L \subseteq K
\]
where $K, L \in F(\ca{X})$ and its strict part $<$ are \boldpii.
\end{lemma}

\begin{proof}
It is clear that
\[
L \subseteq K \eq (\forall x)[x \in L \ \longrightarrow \ x \in K],
\]
so $\leq$ is in \boldpii. For its strict part we have that
\begin{eqnarray*}
K < L &\eq& L \subseteq K \ \& \ L \neq K\\
      &\eq& L \subseteq K \ \& \ K \cap \ca{X} \setminus L \neq \emptyset\\
      &\eq& L \subseteq K \ \& \ K \neq \emptyset \ \& \ (\exists n)[d_n(K) \not \in L].
\end{eqnarray*}
The last equivalence holds because the sequence $(d_n(K))_{\n}$ is dense in $K$ and the set $L$ is closed.

This shows that the strict relation $<$ is in \boldpii.
\end{proof}

\subsection*{The main result.} We state our main result and then we prove it in steps using intermediate lemmas which are interesting in their own right. Since we have already pointed out that we are working inside \ZFDC \ we refrain ourselves from starting every statement below with the phrase ``The following is provable in \ZFDC". 

\begin{theorem}
\label{theorem minimal weakly compact convex}
Suppose that \ca{X} is a separable Banach space and that $F$ is a non-empty, convex and weakly compact subset of \ca{X}. For every Borel-measurable function $T: F \to F$ there is a non-empty convex weakly compact $L \subseteq F$ such that $T[L] \subseteq L$ and moreover $L$ is minimal with respect to these properties.
\end{theorem}

Besides the (semi-)uniformization property of $\boldpii$ the heart of the proof lies also in the following result.

\begin{lemma}
\label{lemma weakly compact is in pii}
Suppose that \ca{X} is a separable Banach space. The set $\ca{R} \subseteq F(\ca{X})$ defined by
\[
K \in \ca{R} \eq K \ \textup{is weakly compact},
\]
is \boldpii.
\end{lemma}

A key tool is the following result of Kleene (\cf \cite{kleene_quantification_numbertheoretic_functions} and for a more modern version 4D.3 in \cite{yiannis_dst}).

\begin{theorem}[The Theorem on Restricted Quantification]
\label{theorem on restricted quantification}

Let $\mathcal{X}$ and $\mathcal{Y}$ be recursively representable metric spaces and let
$Q \subseteq \mathcal{X} \times \mathcal{Y}$ be in $\pii(\ep)$ for
some $\ep \in \mathcal{N}$. Then the set $P \subseteq \ca{X}$ which is defined by
\[
P(x) \eq (\exists y \in \del(\ep,y))Q(x,y).
\]
is also in $\pii(\ep)$.
\end{theorem}

We say that a sequence functions $(f_n)_{\n}$ from a recursively presented metric space \ca{X} to \R \ is \emph{$\Gamma$-recursive} (where $\Gamma$ is a pointclass)  if the relation $P \subseteq \ca{X} \times \om \times \om$ defined by
\[
P(x,n,s) \eq f_n(x) \in N(\R,s)
\]
is in $\Gamma$, where $(N(\R,s))_{s \in \om}$ is a fixed recursive enumeration of all intervals with rational endpoints.

Another important tool is the following result of Debs \cf \cite{debs_effective_properties_in_compact_sets}, which in turn is the effective version of a theorem of Bourgain, Fremlin and Talagrand, \cf \cite{bourgain_fremlin_talagrand_pointwise_compact_sets}.

\begin{theorem}[Debs]
\label{theorem debs}
Suppose that \ca{Y} is a recursively presented Polish space and that $f_n: \ca{Y} \to \R$, \n, are such that the sequence $(f_n)_{\n}$ is   pointwise-bounded and in $\del(\alpha)$ for some parameter $\alpha \in \ca{N}$. If every cluster point of $(f_n)_{\n}$ in $\R^\ca{X}$ \tu{(}with the product topology\tu{)} is a Borel measurable function then there exists some infinite $L \subseteq \om$ in $\del(\alpha)$ such that the subsequence $(f_n)_{n \in L}$ is pointwise convergent.
\end{theorem}

Finally we need the following result of \cite{gregoriades_dichotomy_pointwise_summability} \cf Theorem 2.3.
\begin{theorem}

\label{theorem the week limit is in del}

Suppose that \ca{X} is a separable Banach space and that $\ca{X}$ and $\ca{X}^\om$ are recursively presented. If $(x_n)_{\n} \in \ca{X}^\om$ weakly converges to $x \in \ca{X}$ then $x$ is a $\del((x_n))$ points.
\end{theorem}

\begin{proof}[\textit{Proof of Lemma \ref{lemma weakly compact is in pii}}.]
Let us consider the unit ball $B_{\ca{X}^*}$ of the first dual of $\ca{X}$ with the weak$^*$ topology, that is the least topology on $B_{\ca{X}^*}$ under which every function of the form $x^* \in B_{\ca{X}^*} \mapsto x^*(x)$, where $x \in \ca{X}$, is continuous. Since $\ca{X}$ is separable it can be proved in the context of \ZFDC \  that the space $(B_{\ca{X}^*},{\rm weak}^*)$ is compact Polish (Banach-Alaoglu Theorem in separable spaces). From now and on we will always consider $B_{\ca{X}^*}$ with the weak$^*$ topology.

Let us see first how far we can go using only classical (i.e., non-effective) means. From the Eberlein-\v{S}mulian Theorem -which is a theorem of \ZFDC- we have that
\begin{eqnarray}
\label{equivalence weakly compact}
&& \ \ K \ \textrm{is weakly compact}\\
\nonumber &\eq& \ (\forall (x_n)_{\n} \subseteq K)(\exists L \in [\om]^\om)(\exists x \in K)(\forall x^* \in B_{\ca{X^*}})[\lim_{n \in L}x^*(x_n)=x^*(x)]
\end{eqnarray}
for all $K \subseteq \ca{X}$. (The quantification of the sequence $(x_n)_{\n}$ is done over the Polish space $\ca{X}^\om$.) The latter equivalence only shows that the set
\[
\ca{R} = \set{K \in F(\ca{X})}{K \ \textrm{is weakly compact}}
\]
is a $\tboldsymbol{\Pi}^1_3$ subset of $F(\ca{X})$.\footnote{\label{foonote mention of Moschovakis Uniformization}Here it is worth pointing out Moschovakis' Uniformization Theorem (\cf \cite{yiannis_uniformization_in_a_playfull_universe} or 6C.6 in \cite{yiannis_dst}) which implies that $\ZFDC + \Det{\tboldsymbol{\Delta}^1_2}$ proves that the pointclass $\tboldsymbol{\Pi}^1_3$ has the uniformization property, where  $\Det{\tboldsymbol{\Delta}^1_2}$ is the statement that ``every $\tboldsymbol{\Delta}^1_2$ Gale-Steward game on \om \ is determined". This shows that our strategy does not contradict \ZFDC \ at least. In Footnote \ref{footnote explanation for Pi13} we explain this claim in more detail.}

Now we use methods from effective descriptive set theory to prove that the latter set is in fact \boldpii.

We fix a recursive enumeration $(r_s)_{s \in \om}$ of the rational numbers and a norm dense sequence $(d_i)_{\iin}$ in \ca{X}. Also we consider the fixed recursive enumeration  $(N(\R,s))_{s \in \om}$ of the basis of \R.

Define the function
\[
f: \om \times B_{\ca{X}^*} \to \R: f(i,x^*) = x^*(d_i).
\]
Clearly the function $f$ is continuous and so the set $V \subseteq \om^2 \times B_{\ca{X^*}}$ defined by $$V(n,s,x^*) \eq x^*(d_i) \in N(\R,s)$$ is open. It follows that $V$ is in $\Sigma^0_1(\ep_0)$ for some parameter $\ep_0$. The latter means that the function $f$ is $\ep_0$-recursive. We also consider a parameter $\ep_1$ such that all of the following spaces $\ca{X}$, $\ca{X}^\om$, $B_{\ca{X}^*}$ and $F(\ca{X})$ admit an $\ep_1$-recursive presentation. Moreover we may assume that we have taken the sequence $(d_i)_{\iin}$ as the $\ep_1$-recursive presentation for \ca{X}. As we mentioned before the set $\Mem \subseteq \ca{X} \times F(\ca{X})$ defined by $\Mem(x,F) \eq x \in F$, is Borel and using again the method of relativization we may choose some $\ep_2 \in \ca{N}$ such that $\Mem$ is in $\del(\ep_1,\ep_2)$. Now take $\ep = \langle \ep_0,\ep_1,\ep_2 \rangle$ so that the previous assertions remain valid if we replace $\ep_0$, $\ep_1$ and $\ep_2$ by $\ep$. Finally notice that the ``projection" function $\pr: \om \times \ca{X}^\om \to \ca{X}: \pr(i,(x_n)) = x_i$, is $\ep$-recursive.

For every $x \in \ca{X}$ we consider the function
\[
\tau_x: B_{\ca{X}^*} \to \R: \tau_x(x^*)=x^*(x).
\]
As pointed out before we view every sequence $(x_n)_{\n}$ in \ca{X} as a point of $\ca{X}^\om$. The claim is that for every sequence $(x_n)_{\n}$ the sequence of functions $(\tau_{x_n})_{\n}$ is in $\del(\ep,(x_n))$ in the sense of the theorem of Debs. To see this we verify first the following equivalence
\[
\tau_{x_n}(x^*) < r_s \eq (\exists k \in \om)(\exists i \in \om)[\norm{d_i-x_n} < \frac{1}{k+1} \ \& \ x^*(d_i) < r_s - \frac{2}{k+1}]
\]
for all \n, $x^* \in B_{\ca{X}^*}$ and $s \in \om$, where \norm{\cdot} is the norm of \ca{X}.

For the left-to-right-hand direction choose a $k$ and $i$ such that $\ds \frac{3}{k+1} < r_s - x^*(x_n)$ and $\norm{d_i-x_n} < \ds \frac{1}{k+1}$. Then it is clear that $x^*(d_i) < \ds \frac{1}{k+1} + x^*(x_n) < r_s - \frac{2}{k+1}$. The inverse direction is straightforward.

Since the functions $f = ((i,x^*) \mapsto x^*(d_i))$ and $\pr = ((n,(x_k)) \mapsto x_n)$ are $\ep$-recursive it follows that the condition on the right side of the previous equivalence defines a $\Sigma^0_1(\ep,(x_n))$ subset of $\om^2 \times B_{\ca{X^*}}$. One proves a similar equivalence for the condition $\tau_{x_n}(x^*) > r_s$ and so the the sequence $(\tau_{x_n})_{\n}$ is in fact $\Sigma^0_1(\ep,(x_n))$-recursive.

Now we go back to the equivalence (\ref{equivalence weakly compact}) and we show that the $L$ and $x$ which appear there can be chosen to be in $\del(\ep,(x_n))$ i.e., we claim that
\begin{eqnarray}
\label{equivalence effective weakly compact}
&& \ \ K \ \textrm{is weakly compact}\\
\nonumber &\eq& \ (\forall (x_n)_{\n} \subseteq K)(\exists L \in \del(\ep,(x_n)))(\exists x \in K \cap \del(\ep,(x_n)))\\
\nonumber  &&  \  (\forall x^* \in B_{\ca{X^*}})[\lim_{n \in L}x^*(x_n)=x^*(x)]
\end{eqnarray}
for all $K \subseteq \ca{X}$. The right-to-left-hand implication is clear from (\ref{equivalence weakly compact}) so let us prove the left-to-right-hand implication. Suppose that $(x_n)_{\n}$ is a sequence in $K$. Since the latter set is weakly compact, it is in particular bounded and so the sequence of functions $(\tau_{x_n})_{\n}$ is pointwise bounded. Moreover from the preceding claim the sequence $(\tau_{x_n})_{\n}$ is in $\del(\ep,(x_n))$. Finally since $K$ is weakly compact every cluster point of $(\tau_{x_n})_{\n}$ in $\R^{B_{\ca{X}^*}}$ is a function of the form $\tau_x$ for some $x \in K$. Hence every cluster point of $(\tau_{x_n})_{\n}$ is in fact a continuous function. Therefore all conditions of Debs' Theorem are satisfied and so there exists an infinite $L \subseteq \om$ in $\del(\ep,(x_n))$ such that the subsequence $(\tau_{x_n})_{n \in L}$ is pointwise convergent. Using again the weak compactness of $K$ it follows that the sequence $(x_n)_{n \in L}$ is weakly convergent to some $x \in K$. From Theorem \ref{theorem the week limit is in del} we have that $x$ is in $\del(\ep,(x_n)_{n \in L})$. It is clear that the sequence $(x_n)_{n \in L}$ is recursive in the pair $(L,(x_n)_{n \in \om})$. Since $L$ is in $\del(\ep,(x_n))$ we have that $x$ is in $\del(\ep,(x_n))$ as well. This completes the proof of the equivalence (\ref{equivalence effective weakly compact}). (The application of Debs' Theorem found here is similar to the argument for proving Corollary 1.10 in \cite{gregoriades_dichotomy_pointwise_summability}, p. 161-162.)

Finally we verify that the condition of the right side of (\ref{equivalence effective weakly compact}) defines a \boldpii \ set. The relation $R \subseteq \ca{X}^\om \times \ca{X} \times B_{\ca{X}^*} \times \om$ defined by
\[
R((y_n),x,x^*,s) \eq |x^*(y_n) - x^*(x)| < r_s
\]
is in $\del(\ep)$, c.f. \cite{gregoriades_dichotomy_pointwise_summability}. From this it follows that the relation $Q \subseteq 2^\om \times \ca{X}^\om \times \ca{X} \times B_{\ca{X}^*}$ defined by
\[
Q(L,(x_n),x,x^*) \eq L \ \textrm{is infinite} \ \& \ \lim_{n \in L}x^*(x_n)=x^*(x)
\]
is also in $\del(\ep)$. Moreover since the relation $\Mem$ is in $\del(\ep)$ it follows that $T \subseteq \om \times \ca{X}^\om \times F(\ca{X})$ defined by
\[
T(i,(x_n),K) \eq x_i \in K
\]
is in $\del(\ep)$ as well. Using the Theorem on Restricted Quantification and the preceding comments we have that the right side of the equivalence (\ref{equivalence effective weakly compact}) defines a set in $\pii(\ep)$ and thus a \boldpii \ set.
\end{proof}

\begin{remark}
\label{remark bossard}\normalfont
It is well-known that the separable Banach space $C(2^\om)$ of the continuous real functions on $[0,1]$ with the maximum norm $\norm{\cdot}_{\infty}$ is universal for the class of separable Banach spaces, \ie every separable Banach space is isometric to a closed subset of $(C(2^\om),\norm{\cdot}_{\infty})$. Therefore we may view the class of all separable Banach spaces as a subset of $F(C(2^\om))$. 

We fix for the rest of this article the set $\REFL \subseteq F(C(2^\om))$ defined by
\[
\REFL(X) \iff X \ \textrm{is reflexive}.
\]
 Bossard \cf \cite{bossard_codages_banach} has proved that the set \REFL \ is \emph{Borel \boldpii-complete}, \ie it is \boldpii \ and that every \boldpii \ subset of a Polish space is reducible to \REFL \ via a Borel function.

Our remark is that one can prove that \REFL \ is a \boldpii \ set using Lemma \ref{lemma weakly compact is in pii}. Recall that a Banach space is reflexive exactly when its closed unit ball is weakly compact, so from Lemma \ref{lemma weakly compact is in pii} it is enough to prove that the mapping
\[
\Ball: F(C(2^\om)) \to F(C(2^\om)): X \mapsto \set{f \in X}{\norm{f}_{\infty} \leq 1}
\]
is Borel-measurable. To prove the latter we consider the sequence of Borel-measurable functions $$d_n: F(C(2^\om)) \to C(2^\om)$$ of Theorem \ref{theorem selection in F(X)} and a non-empty open $U \subseteq C(2^\om)$. Notice that for all $X \in F(C(2^\om))$, $f \in X$ with $\norm{f}_{\infty} > 0$ and $\ep > 0$ there exists some \n \ such that
\[
\norm{f}_{\infty} - \ep < \norm{d_n(X)}_{\infty} < \norm{f}_{\infty}.
\]
Using this remark we have that
\begin{align*}
\Ball(X) \in \ca{A}_U \iff& \Ball(X) \cap \ca{A}_U \neq \emptyset\\
                      \iff& \set{f \in X}{\norm{f}_{\infty} \leq 1} \cap U \neq \emptyset\\
                      \iff& (\exists k, n)[\norm{d_n(X)}_{\infty} < 1 \ \& \ \overline{B}_{\infty}(d_n(X),(k+1)^{-1}) \subseteq U]\\
                      \iff& (\exists k, n)(\forall m)\big \{\norm{d_n(X)}_{\infty} < 1\\
                          & \hspace*{3mm}\& \ [\norm{d_m(X)-d_m(X)}_{\infty} < (k+1)^{-1} \longrightarrow \Mem(d_m(X),U)]\big \}.
\end{align*}
Therefore the function $\Ball$ is Borel-measurable.
\end{remark}

Now we get back to our proof. We say that a family \ca{R} of subsets of some set $X$ is \emph{closed under intersections of $\subseteq$-chains} if for all $(A_i)_{i \in I}$ in \ca{R}, for which $A_i \subseteq A_j$ or $A_j \subseteq A_i$ for all $i, j \in I$, the intersection $\cap_{i \in I} A_i$ is a member of \ca{R}.

Some of the arguments that we are going to use in the proof of Theorem \ref{theorem minimal weakly compact convex} yield the following result which is worth pointing out.

\begin{lemma}
\label{lemma general minimal with respect to a good property}
Suppose that \ca{X} is a Polish space, $F$ is a non-empty closed subset of \ca{X} and that \pntcl \ is a good pointclass which has the semi-uniformization property and contains \boldpii. Moreover suppose that $\mathcal{R} \subseteq F(\ca{X})$ satisfies the following:

\tu{(1)} $F \in \ca{R}$,

\tu{(2)} \ca{R} is in \pntcl,

\tu{(3)} \ca{R} is closed under intersections of $\subseteq$-chains.

\noindent
Then for every Borel measurable function  $T: F \to F$ there is a minimal $L \subseteq F$ which is $T$-invariant and belongs to \ca{R}, i.e., there exists some $L \subseteq F$ in $\ca{R}$ such that $T[L] \subseteq L$ and for which there is no $L' \subseteq F$ which belongs to \ca{R} with $T[L'] \subseteq L'$ and $L' \subsetneq L$.
\end{lemma}

(Notice that we do not exclude the possibility $L = \emptyset$. Of course the interesting applications are when membership in \ca{R} excludes this possibility, as it is in the case of Theorem \ref{theorem minimal weakly compact convex}.)

\begin{proof}
Define the set $\mathbb{P} \subseteq F(\ca{X})$ by
\[
K \in \mathbb{P} \eq K \subseteq F \ \& \ K \in \ca{R} \ \& \ T[K] \subseteq K.
\]
Notice that $\mathbb{P}$ is not empty, because $F$ is a member of \ca{R} and $T$ takes values inside $F$. Using the fact that $T$ is Borel-measurable we have that the set $\ca{R}_1 \subseteq F(\ca{X})$ defined by
\[
K \in \ca{R}_1 \eq T[K] \subseteq K,
\]
is in $\boldpii \subseteq \pntcl$. Since \ca{R} is also in \pntcl \ and \pntcl \ is good, it follows that $\mathbb{P}$ is in \pntcl \ as well.

We consider the partial order $\leq$ on $F(\ca{X})$ as defined in Lemma \ref{lemma less-or-equal is in pii}. As we have shown in the preceding lemma the strict relation $<$ is in \boldpii \ and so it is in \pntcl. In particular $<$ is computed by \pntcl \ on $\mathbb{P} \times \mathbb{P}$.

We now show that the space $(\mathbb{P},\leq)$ is inductive. Indeed suppose that $(K_i)_{i \in I}$ is a $\leq$-chain in $\mathbb{P}$. Consider the set $K : = \cap_{i \in I} K_i$. From our hypothesis about \ca{R} the set $K$ is a member of \ca{R}. Moreover $K \subseteq F$ and $T[K] \subseteq \cap_{i \in I} T[K_i] \subseteq \cap_{i \in I} K_i = K$. Thus $K \in \mathbb{P}$. From the definition of $K$ it is clear that $K = \sup_{\leq}\set{K_i}{i \in I}$.

Now we apply Lemma \ref{lemma inductive space maximal point} to get a $\leq$-maximal point $L \in \mathbb{P}$. It is clear that this $L$ satisfies the conclusion.
\end{proof}

\begin{proof}[\textit{Proof of Theorem \ref{theorem minimal weakly compact convex}}.]
We will apply Lemma \ref{lemma general minimal with respect to a good property} with $\pntcl = \tboldsymbol{\Pi}^1_1$, which has the uniformization property, and \ca{R} being defined by
\[
K \in \ca{R} \eq K \ \textrm{is non-empty, convex and weakly compact}.
\]
Recall that every weakly closed set is also closed in the norm topology, so $\ca{R} \subseteq F(\ca{X})$. We show that conditions (1)-(3) in the statement of Lemma \ref{lemma general minimal with respect to a good property} are satisfied. It is clear that $F \in \ca{R}$ and that \ca{R} is closed under intersections of $\subseteq$-chains, so conditions (1) and (3) are satisfied. We now show that \ca{R} is in $\tboldsymbol{\Pi}^1_1$ to meet condition (2). Consider the sets $\ca{R}_i \subseteq F(\ca{X})$, $i=1,2,3$ defined by
\begin{eqnarray*}
K \in \ca{R}_1 &\eq& K \neq \emptyset\\
K \in \ca{R}_2 &\eq& K \ \textrm{is convex}\\
K \in \ca{R}_3 &\eq& K \ \textrm{is weakly compact}
\end{eqnarray*}
for $K \in F(\ca{X})$. It is enough to show that the sets $\ca{R}_i$, $i=1,2,3$, are in \boldpii. We remark that
\[
K \in \ca{R}_1 \eq K \in A_\ca{X},
\]
so $\ca{R}_1$ is in fact a Borel subset of $F(\ca{X})$. Regarding $\ca{R}_2$, using the fact that we are dealing with closed sets, it follows that
\begin{align*}
K \in \ca{R}_2 \iff& (\forall x,y \in \ca{X})(\forall t \in [0,1])[x,y \in K \longrightarrow tx+(1-t)y \in K]\\
                       \iff&(\forall n,m)(\forall q \in \mathbb{Q} \cap [0,1])[qd_n(K) + (1-q)d_m(K) \in K]
\end{align*}
where the $d_n$ are the functions from the Kuratowski, Ryll-Nardzewski Theorem. Hence $\ca{R}_2$ is also a Borel set. From Lemma \ref{lemma weakly compact is in pii} we have that $\ca{R}_3$ is a \boldpii \ set as well. So condition (2) is satisfied.
Now we apply Lemma \ref{lemma general minimal with respect to a good property} and we get a minimal $L \subseteq F$ which is $T$-invariant and belongs to \ca{R}.\footnote{\label{footnote explanation for Pi13}Assuming $\Det{\tboldsymbol{\Delta}^1_2}$ and using Moschovakis' Uniformization Theorem we may choose the pointclass \pntcl \ in Lemma \ref{lemma general minimal with respect to a good property} to be  $\tboldsymbol{\Pi}^1_3$. As remarked in the proof of Lemma \ref{lemma weakly compact is in pii} it is a straightforward consequence of the Eberlein-\v{S}mulian Theorem that the set $\ca{R}_3$ is $\tboldsymbol{\Pi}^1_3$ and therefore the set $\ca{R}$ is $\tboldsymbol{\Pi}^1_3$ as well. So from Lemma \ref{lemma general minimal with respect to a good property} we get our result. Therefore $\ZFDC + \Det{\tboldsymbol{\Delta}^1_2}$ proves Theorem \ref{theorem minimal weakly compact convex} in a relatively easy way, as we do not need to get into the pains of proving that \ca{R} is in fact $\tboldsymbol{\Pi}^1_1$. This remark precedes chronologically  Lemma \ref{lemma weakly compact is in pii} and it served as an indication of the correctness of our strategy.}
\end{proof}

\begin{remark}\normalfont
\label{remark schoenfield absoluteness}A well-known application of \emph{Shoenfield's Absoluteness Theorem} (\cf \cite{jech_set_theory} Theorem 25.20) deals with removing the Axiom of Choice from proofs of statements of certain complexity. It is natural to ask if Shoenfield's Absoluteness can be applied in our case. After some ambivalence from the part of the author, H. Mildenberger has pointed out that this is indeed correct. It is a well-known corollary of the latter that if G\"odel' s constructible universe ${\rm L}$ proves a $\Sigma^1_3$ statement, then this statement holds in any model of \ZFDC. The similar assertion is true with respect to some parameter $\alpha$. Using Lemma \ref{lemma weakly compact is in pii}, one can see that the statement ``there exists a minimal set which satisfies the conclusion in Theorem  \ref{theorem minimal weakly compact convex}'' is a $\tboldsymbol{\Sigma}^1_3$ statement and hence it is $\Sigma^1_3(\alpha)$ for some $\alpha \in \ca{N}$. Since ${\rm L}[\alpha]$ is a model of $\ZFC$, it proves in particular $\varphi$, (because the latter is a theorem of \ZFC). It follows from Shoenfield's Absoluteness that this statement is a theorem of \ZFDC. This remark shows that one can trade Kondo's Uniformization with Shoenfield's Absoluteness in order to derive a minimal set $L$ as above without the Axiom of Choice (in separable Banach spaces).
\end{remark}

In the last pages of this article we present the Banach space theoretic approach for proving Theorem \ref{theorem minimal weakly compact convex}. This approach was suggested by the referee.

By \emph{tree} on a set $X$ we mean a non-empty set $T$ of finite sequences of points in $X$ (including the empty sequence) which is closed downwards under initial segments. Given a tree $T$ we denote with $T_{u}$ the set all finite sequences in $T$ which are compatible with the finite sequence $u$. It is easy to see that $T_{u}$ is a tree as well. By enumerating the set of all finite sequences of naturals and by identifying a tree on \om \ with its characteristic function, we can view a tree as a member of $2^\om$. We denote by \Tr \ the space of all trees on \om. It is easy to see that \Tr \ is a closed subset of $2^\om$ and thus it is a Polish space. A tree is \emph{well-founded} if it has no infinite branches. We denote by \WF \ the set of all well-founded trees on \om. The latter set is the standard example of a $\boldpii$-complete set.

Our aim is to show that the set $\ca{R}$ (in the notation of the Proof of Theorem \ref{theorem minimal weakly compact convex}) Borel-reduces to the set \WF, \ie that there exists a Borel-measurable function
\[
\Phi: F(\ca{X}) \to \Tr
\]
such that $\ca{R} = \Phi^{-1}[\WF]$. Since \WF \ is a \boldpii \ set and $\Phi$ is Borel-measurable, it follows that \ca{R} is \boldpii \ as well.

We proceed with some Banach space theoretic terminology. Suppose that $0 < \ep \leq 1 \leq M$ are given and that $(x_1,\dots,x_m)$ is a finite sequence in the Banach space \ca{X}.

We say that $(x_1,\dots,x_m)$ is \emph{$M$-Schauder}
if for all $k \in \{1,\dots,m\}$ and all real numbers $a_1,\dots,a_m$ we have
\[
\norm{\sum_{n=1}^k a_nx_n} \leq M \cdot \norm{\sum_{n=1}^m a_nx_n}.
\]
The sequence $(x_1,\dots,x_m)$ \emph{\ep-dominates the summing basis} if for all positive reals $a_1, \dots, a_m$ with $\sum_{n=1}^m a_n=1$ we have that
\[
\norm{\sum_{n=1}^ma_nx_n} \geq \ep.
\]
Observe that the members of a finite sequence, which is $M$-Schauder (resp. $\ep$-dominates the summing basis), is bounded from above (resp. from below) by $M$ (resp. by $\ep$) in norm. 

We say that an infinite sequence $(x_n)_{\n}$ is \emph{basic} if $x_m \neq 0$ for all $m$ and there exists some $M > 0$ such that $(x_1,\dots,x_m)$ is $M$-Schauder for all $m$. The preceding condition is known to be equivalent to the statement that $(x_n)_{\n}$ is a Schauder basis of the space $\overline{\rm span}\set{x_n}{\n}$, (\ie the closure in \ca{X} of the space generated by $\set{x_n}{\n}$). Notice that every subsequence of a basic sequence is basic as well.

We follow the technique of \cite{argyros_dodos_genericity_amalgamation_Banach_spaces} pp. 679-680 (see also \cite{dodos_Banach_spaces_and_dst_lecure_notes} Section 2.2). We fix a separable Banach space \ca{X} and the functions $d_n : F(\ca{X}) \to \ca{X}$ of the Kuratowski, Ryll-Nardzewski Theorem.

For all $0 < \ep \leq 1 \leq M$ and all $F \in F(\ca{X})$ we define $T(F,\ep,M)$ as follows
\begin{align*}
(u_1,\dots,u_m) \in T(F,\ep,M) \iff& (d_{u_1}(F),\dots,d_{u_m}(F)) \ \textrm{is $M$-Schauder}\\
                                                      & \hspace*{15mm} \textrm{and $\ep$-dominates the summing basis}
\end{align*}
for all natural numbers $u_1,\dots, u_m$. (By $m=0$ in the definition above we mean the empty sequence.) It is easy to verify that $T(F,\ep,M)$ is a tree on \om.

Let us denote by $\BC(\ca{X})$ the set of all non-empty closed bounded convex subsets of \ca{X} and by $\WCC(\ca{X})$ the set of all $F \in \BC(\ca{X})$ such that $F$ is weakly compact. Clearly $\WCC(\ca{X}) = \ca{R}$ in the notation of the proof of Theorem \ref{theorem minimal weakly compact convex}. Notice that $\BC(\ca{X})=\ca{R}_1 \cap \ca{R}_2 \cap \set{F \in F(\ca{X})}{\textrm{$F$ is bounded}}$ is a Borel subset of $F(\ca{X})$. We now claim the following variation of Lemma 5 in \cite{argyros_dodos_genericity_amalgamation_Banach_spaces}.

\begin{lemma}
\label{lemma variation argyros-dodos}
For all $F \in \BC(\ca{X})$ we have that
\[
F \in \WCC(\ca{X}) \iff ({\rm for \ all} \ 0 < \ep \leq 1 \leq M)[T(F,\ep,M) \in \WF].
\]
\end{lemma}
For reasons of completeness we will give a proof of the preceding lemma, but before we do so let us remark how one can derive Theorem \ref{theorem minimal weakly compact convex} from this result. As before we apply  Lemma \ref{lemma general minimal with respect to a good property} with $\pntcl = \tboldsymbol{\Pi}^1_1$ and the problem is reduced to proving that the set $\WCC(\ca{X})$ is \boldpii.

We define $$\Phi: \BC(\ca{X}) \to \Tr$$ by
\[
u \in \Phi(F) \iff \textrm{if} \ u=(n,u_1,\dots,u_{m-1}) \ \textrm{then} \ (u_1,\dots,u_{m-1}) \in T(F,1/(n+1),n+1).
\]
It is clear that $\Phi(F)$ is a tree on $\om$ and that $\Phi(F)_{(n)} = T(F,1/(n+1),n+1)$ for all \n. By unraveling the definitions one can see that $\Phi$ is Borel-measurable. (Clearly it is enough to quantify the $a_1,\dots,a_n$'s over the rationals.)

It is also clear that if $$0 \leq \ep' \leq \ep \leq 1 \leq M \leq M'$$ and $(x_1,\dots,x_m)$ is $M$-Schauder (resp. $\ep$-dominates the summing basis) then \mbox{} $(x_1,\dots,x_m)$ is $M'$-Schauder (resp. $\ep'$-dominates the summing basis) as well, and so $T(F,\ep,M) \subseteq T(F,\ep',M')$.

Using the preceding remarks and Lemma \ref{lemma variation argyros-dodos} we have that
\begin{align*}
F \in \WCC(\ca{X}) \iff& ({\rm for \ all} \ 0 < \ep \leq 1 \leq M)[T(F,\ep,M) \in \WF]\\
                              \iff& (\forall n)[T(F,1/(n+1),n+1) \in \WF]\\
                              \iff& (\forall n)[\Phi(F)_{(n)} \in \WF]\\
                              \iff& \Phi(F) \in \WF
\end{align*}
for all $F \in \BC(\ca{X})$. Since the latter set is a Borel subset of $F(\ca{X})$, we can extend $\Phi$ to a Borel-measurable function $\tilde{\Phi}$ on $F(\ca{X})$ in such a way that $\tilde{\Phi}(F) \not \in \WF$ for all $F \not \in \BC(\ca{X})$. Hence
\[
F \in \WCC(\ca{X}) \iff \tilde{\Phi}(F) \in \WF
\]
for all $F \in F(\ca{X})$ and so $\WCC(\ca{X}) = \tilde{\Phi}^{-1}[\WF]$ is a $\boldpii$ set.

It remains to prove Lemma \ref{lemma variation argyros-dodos}. Let $F$ be a member of $\BC(\ca{X})$. Suppose that $0 < \ep \leq 1 \leq M$ are given and that the tree $T = T(F,\ep,M)$ is not well-founded. We consider an infinite branch $\alpha: \{1,2,\dots \} \to \om$ of $T$ and let $x_n = d_{\alpha(n)}(F)$ for all $n \geq 1$. By Rosenthal's $\ell_1$ Dichotomy Theorem (\cf \cite{kechris_classical_dst} Theorem 19.20) there exists some $\{1 \leq l_0 < l_1 < \dots \} \subseteq \om$ such that the subsequence $(x_{l_n})_{\n}$ is either equivalent to the standard unit vector basis of $\ell_1$ or it is weak$^\ast$ convergent to some $x^{\ast \ast}$ in the second dual $\ca{X}^{\ast \ast}$. In the first case we immediately get that $(x_{l_n})_{\n}$ does not have a weakly convergent subsequence in $\ca{X}$ -and hence in $F$. Therefore $F$ is not weakly compact. In the second case we consider whether $x^{\ast \ast}$ belongs to $\ca{X}$ or not. If $x^{\ast \ast}$ is not a member of $\ca{X}$ then again $(x_{l_n})_{\n}$ does not have a weakly convergent subsequence in $\ca{X}$ and as before $F$ is not weakly compact. Now we show that the subcase $x^{\ast \ast} \in \ca{X}$ cannot occur. This will complete the proof of the direct implication of Lemma \ref{lemma variation argyros-dodos}. Suppose towards contradiction that $x^{\ast \ast} = x$ is a member of $\ca{X}$. We take \ca{Y} to be the closure of the space generated by $\set{x_{l_n}}{\n}$. Then $x$ is a member of \ca{Y} and the sequence $(x_{l_n})_{\n}$ is a Schauder basis of $\ca{Y}$. By taking the diagonal functionals we obtain that $x = 0$, \ie the sequence $(x_{l_n})_{\n}$ is weakly null. From Mazur's Theorem there exists a finite convex combination $z$ of $\set{x_{l_n}}{\n}$ such that $\norm{z} < \ep$, contradicting that $(x_1,\dots,x_N)$ $\ep$-dominates the summing basis for all $N \geq 1$.

Conversely assume that $F$ is not weakly compact. We will show that for some $\ep$ and $M$ the tree $T(F,\ep,M)$ is not well-founded. Notice that the weak closure of $F$ coincides with its norm closure -which is $F$- since $F$ is a convex set (Mazur's Theorem). We apply the Pe{\l}czy{\'n}ski form of the Eberlein-\v{S}mulian Theorem to $F$, \cf \cite{pelczynski_proof_Eberlein_Smulian_application_basic_sequences} or \cite{diestel_sequences_series_Banach_spaces} p. 41. The latter implies that $F$ contains a basic sequence $(z_n)_{\n}$ of elements of $F$ and that there exists some $x^\ast$ in the closed unit ball of $\ca{X}^\ast$ such that $\lim_{n} x^\ast(z_n) \geq 2\ep > 0$ for some $\ep < 1$. By removing an initial segment of the sequence $(z_n)_{\n}$ we may assume that $x^\ast(z_n) \geq 2\ep$ for all \n. The next remark is that we can approximate every $z_n$ by some $d_{u_n}(F)$ such that the sequence $(d_{u_n}(F))_{\n}$ is basic (in fact equivalent to $(z_n)_{\n}$) and $x^\ast(d_{u_n}(F)) \geq \ep$ for all \n, see Theorem 1 in \cite{bessaga_pelczynski_generalization_results_James}. Put $y_n = d_{u_n}(F)$ for all \n. For all positive real numbers $a_1, \dots, a_m$ with $\sum_{n=1}^m a_n = 1$ we have that
\[
\norm{\sum_{n=1}^m a_n y_n} \geq x^\ast(\sum_{n=1}^m a_n y_n) = \sum_{n=1}^m a_n x^\ast(y_n) \geq \ep \sum_{n=1}^m a_n = \ep.
\]
Hence the finite sequence $(y_1,\dots,y_m)$ $\ep$-dominates the summing basis for all $m$. Since $(y_n)_{\n}$ is basic we can choose some $M \geq 1$ such that $(y_1,\dots,y_m)$ is $M$-Schauder for all $m$. This implies that $(u_1,u_2,\dots,u_n,\dots)$ is an infinite branch of $T(F,\ep,M)$. The proof is complete.

\end{document}